\documentclass[11pt,a4paper]{amsart}
\usepackage[T1]{fontenc}
\usepackage{lmodern}
\usepackage{amsmath,amssymb,mathtools,microtype}
\usepackage[hmargin=27mm,vmargin=25mm]{geometry}
\usepackage[hidelinks]{hyperref}
\newtheorem{theorem}{Theorem}[section]
\newtheorem{lemma}[theorem]{Lemma}
\newtheorem{proposition}[theorem]{Proposition}
\newtheorem{corollary}[theorem]{Corollary}
\theoremstyle{remark}
\newtheorem{remark}[theorem]{Remark}
\newcommand{\R}{\mathbb R}
\newcommand{\E}{\mathcal E}
\newcommand{\one}{\mathbf 1}
\newcommand{\Cor}{\mathcal C}
\newcommand{\dev}{\operatorname{dev}}
\title[Local quadratic convergence of iterated correlation matrices]{Local quadratic convergence of iterated correlation matrices: bounds and counterexamples}
\author{Ishrak Alhajj Hassan}
\address{University of Ostrava, Czech Republic}
\date{October 17, 2026}
\begin{document}
\begin{abstract}
Kruskal's unpublished manuscript on CONCOR gives a local argument for convergence of successive correlations to a prescribed two-class sign pattern. We identify a failure in its stated one-step estimate and establish a replacement bound for correlation matrices in arbitrary dimension. A geometric estimate for sign-aligned centered rows yields an explicit sufficient neighbourhood of convergence and a quadratic error bound. We construct invariant families of correlation matrices with an exact scalar recurrence in every dimension at least three. These families yield counterexamples to the original estimate in every dimension at least four. The failure persists on relatively open sets of positive-definite correlation matrices. In dimensions divisible by four, an explicitly solvable subfamily provides counterexamples arbitrarily close to the prescribed fixed point and attains the asymptotic coefficient in the new bound. Quadratic order is attained by the constructed families in every dimension at least three. These counterexamples concern the quantitative estimate, not the convergence conclusion: all the exhibited trajectories converge to their prescribed sign patterns.
\end{abstract}
\maketitle

\section{Introduction and main results}
Repeatedly replacing a matrix by the correlations between its rows defines the iteration underlying CONCOR. The procedure was developed for relational data by Breiger, Boorman, and Arabie~\cite{BBA}; Chen~\cite{Chen} studied iteratively generated correlation matrices in connection with visualization, clustering, and seriation. A natural candidate for a limit is a matrix
\[
P^*=\sigma\sigma^{\mathsf T},\qquad \sigma\in\{-1,1\}^n,
\]
where both signs occur. Its entries are $1$ within each of two nonempty classes and $-1$ between them.

Schwartz~\cite[pp.~258--266]{Schwartz} analyzed iterated covariations through matrix powers and spectral decomposition. For Pearson correlation iteration, the changing diagonal normalization prevents the same argument from applying directly; his discussion explicitly leaves convergence unproved~\cite[pp.~273--276]{Schwartz}. He also exhibited non-sign-valued fixed points in dimension three~\cite[pp.~280--281]{Schwartz}. The complete fixed-point and stability classification, together with global convergence of every admissible trajectory in dimension three, is established in~\cite{AH}. The present paper addresses the complementary quantitative question of convergence rates near a prescribed sign-pattern fixed point in arbitrary dimension.

In the scanned manuscript \emph{A theorem about CONCOR}, dated September~21, 1977, Kruskal~\cite[pp.~1--3]{Kruskal} stated a sufficient condition for convergence near such a prescribed pattern. Chen~\cite[\S3.2]{Chen} discusses this work and lists it as an unpublished manuscript. In their study of iterative distance-profile clustering, Dubnov et al.~\cite[p.~57]{Dubnov} reproduce Kruskal's sufficient condition for Pearson iteration and cite the 1978 technical report.

Kruskal's argument measures the maximum entrywise deviation from a prescribed sign pattern and seeks a one-step bound quadratic in that deviation~\cite[p.~2]{Kruskal}. His convergence argument then uses a contraction factor below one~\cite[pp.~3, 9--10]{Kruskal}. We retain this framework and correct its quantitative estimate.

We prove three main results. First, a geometric bound for sign-aligned centered rows gives a valid one-step estimate and an explicit sufficient neighbourhood of quadratic convergence in every dimension $n\geq2$. Second, repeated Gram vectors produce invariant families with an exact scalar recurrence in every dimension $n\geq3$. These families disprove Kruskal's estimate in every dimension $n\geq4$, and the failure persists on relatively open sets of positive-definite correlation matrices. Third, the families attain quadratic order in every dimension $n\geq3$. For balanced partitions in dimensions divisible by four, an explicitly solvable subfamily gives counterexamples arbitrarily close to the fixed point and attains the asymptotic coefficient $1/e=1/2$, proving that this coefficient cannot be reduced in those dimensions.

\subsection*{Notation}
Let $\E_n$ denote the set of real symmetric positive semidefinite $n\times n$ matrices with unit diagonal. Their entries belong to $[-1,1]$. Set
\[
H=I_n-\frac1n\one\one^{\mathsf T}.
\]
For the rows $p_i=P(i,:)$, write $x_i=p_iH$. If all $x_i$ are nonzero, define the row-correlation map by
\begin{equation}\label{eq:map}
\Cor(P)_{ij}=\frac{\langle x_i,x_j\rangle}{\|x_i\|\,\|x_j\|}.
\end{equation}
Here $\one\in\R^n$ is the all-ones vector. A matrix is called Pearson-nondegenerate if all its centered rows are nonzero; an admissible initial correlation matrix is an element of $\E_n$ with this property. All vector norms are Euclidean. The matrix $\Cor(P)$ is a Gram matrix of unit vectors, so it belongs to $\E_n$. Kruskal~\cite[p.~2]{Kruskal} formulates his iteration using column correlations and allows general initial matrices with entries in $[-1,1]$. We state the main results for correlation matrices, for which row and column correlations agree.

Following Kruskal's deviation and class-size normalization~\cite[pp.~1--2]{Kruskal}, fix $\sigma\in\{-1,1\}^n$ with both signs present, and put
\begin{equation}\label{eq:parameters}
\dev(P)=\max_{i,j}|P_{ij}-\sigma_i\sigma_j|,
\qquad e=\frac{8m(n-m)}{n^2},
\end{equation}
where $m$ is the smaller sign-class size. Thus $0<e\leq2$. A partition is \emph{balanced} if both classes have equal size; then $e=2$. The pattern is fixed whenever $\dev$ is used. Write $\eta=\dev(P)$ and $\eta^+=\dev(\Cor(P))$ for the deviations before and after one iteration, respectively. Its centered rows are $\sigma_i\sigma^{\mathsf T}H\ne0$, and hence $P^*$ is a fixed point of~\eqref{eq:map}.

\begin{theorem}[Local quadratic convergence]\label{thm:main}
Let $n\geq2$, $P_0\in\E_n$, and let $\sigma$ contain both signs. Define
\begin{equation}\label{eq:gamma}
\gamma(e)=\frac{2(e+1-\sqrt{2e+1})}{e}
=\frac{2e}{e+1+\sqrt{2e+1}}.
\end{equation}
If $\eta_0=\dev(P_0)<\gamma(e)$, all iterates $P_{t+1}=\Cor(P_t)$ are well defined and converge to $P^*$. If $\eta_0>0$, put
\begin{equation}\label{eq:Cq}
C=\frac{1}{e(1-\eta_0/2)^2},\qquad q=C\eta_0<1.
\end{equation}
Then, for all $t\geq0$,
\begin{equation}\label{eq:rate}
\eta_{t+1}\leq C\eta_t^2,
\qquad \eta_t\leq C^{-1}q^{\,2^t},
\qquad \eta_t:=\dev(P_t).
\end{equation}
If no iterate equals $P^*$, then
\begin{equation}\label{eq:asymptotic}
\limsup_{t\to\infty}\frac{\eta_{t+1}}{\eta_t^2}\leq\frac1e.
\end{equation}
\end{theorem}

For a balanced partition, the sufficient threshold is
\[
\gamma(2)=3-\sqrt5\approx0.763932.
\]
The theorem is local: the limiting pattern is specified in advance. It does not determine which pattern a general trajectory selects. In dimension three, every admissible trajectory converges to one of seven fixed points~\cite{AH}. That global result concerns all admissible initial states; the present theorem gives a quantitative rate near a prescribed sign-pattern fixed point in arbitrary dimension.

\begin{theorem}[Counterexamples and attained rates]\label{thm:examples}
For every $n\geq4$, there exist $P\in\E_n$ and a sign pattern with two classes of equal size, or sizes differing by one, such that, with $\eta=\dev(P)=1/3$,
\begin{equation}\label{eq:counterclaim}
\dev(\Cor(P))>\frac{\eta^2}{e(1-\eta^2)}.
\end{equation}
The trajectories of these matrices nevertheless converge to the prescribed sign patterns. The failure of the estimate also occurs on nonempty relatively open sets of positive-definite correlation matrices whose trajectories converge to the same patterns. In every dimension divisible by four, counterexamples also exist arbitrarily close to a balanced sign-pattern fixed point, and a convergent family attains the coefficient $1/e=1/2$ in~\eqref{eq:asymptotic}. In every dimension $n\geq3$, a convergent family attains quadratic order.
\end{theorem}

Here a trajectory attains quadratic order if its positive deviations satisfy $\eta_{t+1}/\eta_t^2\to K$ for some $K\in(0,\infty)$.

Theorem~\ref{thm:main} is proved in Section~\ref{sec:bound}. The construction and proof of Theorem~\ref{thm:examples} occupy Sections~\ref{sec:families} and~\ref{sec:sharp}.

\section{The original estimate and a four-dimensional counterexample}
The one-step estimate stated in Kruskal's manuscript~\cite[p.~2]{Kruskal}, in the notation~\eqref{eq:parameters}, is
\begin{equation}\label{eq:old}
\dev(\Cor(P))\leq\frac{\eta^2}{e(1-\eta^2)}.
\end{equation}
The calculation in~\cite[pp.~7--8]{Kruskal} involves centered-norm lower bounds containing $(1-\eta)^2$. Replacing that factor by $1-\eta^2$ does not follow from the preceding bounds. More decisively, the following correlation matrix violates~\eqref{eq:old}.

Take $\sigma=(1,1,-1,-1)^{\mathsf T}$ and
\begin{equation}\label{eq:four}
P_0=\begin{pmatrix}
1&2/3&-1&-2/3\\
2/3&1&-2/3&-1\\
-1&-2/3&1&2/3\\
-2/3&-1&2/3&1
\end{pmatrix}.
\end{equation}
Its eigenvalues are $10/3,2/3,0,0$. Its rows have mean zero and squared norm $26/9$, so
\[
\Cor(P_0)=\frac9{26}P_0^2
=\begin{pmatrix}
1&12/13&-1&-12/13\\
12/13&1&-12/13&-1\\
-1&-12/13&1&12/13\\
-12/13&-1&12/13&1
\end{pmatrix}.
\]
Here $e=2$ and $\eta_0=1/3$, whereas
\[
\dev(\Cor(P_0))=\frac1{13}>\frac1{16}
=\frac{\eta_0^2}{e(1-\eta_0^2)}.
\]
This is an exact counterexample. It does not contradict convergence: $\eta_0<\gamma(2)$, so Theorem~\ref{thm:main} applies. Section~\ref{sec:sharp} also gives its trajectory explicitly.

\section{A geometric bound and quadratic convergence}\label{sec:bound}
The interval information near a sign pattern places all sign-aligned centered rows in a common Euclidean ball. Controlling their directions through this ball gives a stronger estimate than combining separate distance and norm bounds.

\begin{lemma}[Directions in a ball]\label{lem:ball}
Let $c\ne0$ and let $u,v\ne0$ satisfy $\|u-c\|\leq R$ and $\|v-c\|\leq R$. Then
\begin{equation}\label{eq:ball}
1-\frac{\langle u,v\rangle}{\|u\|\,\|v\|}
\leq\frac{2R^2}{\|c\|^2}.
\end{equation}
\end{lemma}
\begin{proof}
If $R\geq\|c\|$, the right side is at least $2$, and the assertion follows from Cauchy--Schwarz. Suppose $R<\|c\|$, and set $z=R/\|c\|<1$ and $w=c/\|c\|$. The ball lies in the half-space $\langle x,c\rangle>0$. The distance from $c$ to the line spanned by $u$ is at most $R$, so, writing $\widehat u=u/\|u\|$, we have
\[
\langle\widehat u,w\rangle\geq\sqrt{1-z^2}.
\]
The same holds for $\widehat v=v/\|v\|$. Decompose $\widehat u=\alpha w+u_\perp$ and $\widehat v=\beta w+v_\perp$. Then $\alpha,\beta\geq\sqrt{1-z^2}$ and $\|u_\perp\|,\|v_\perp\|\leq z$. Consequently
\[
\langle\widehat u,\widehat v\rangle
\geq\alpha\beta-\|u_\perp\|\,\|v_\perp\|
\geq1-2z^2,
\]
which is~\eqref{eq:ball}.
\end{proof}

\begin{proposition}[One-step bound]\label{prop:bound}
Let $P\in\E_n$ and $\eta=\dev(P)<1$. Then $\Cor(P)$ is well defined and
\begin{equation}\label{eq:new}
\dev(\Cor(P))\leq\frac{\eta^2}{e(1-\eta/2)^2}.
\end{equation}
\end{proposition}
\begin{proof}
Write $v_i=\sigma_iP(i,:)$ and $u_i=v_iH$. The definition of $\eta$ and the entry bounds for correlation matrices give
\begin{equation}\label{eq:intervals}
v_{i\ell}\in
\begin{cases}
[1-\eta,1],&\sigma_\ell=1,\\
[-1,-1+\eta],&\sigma_\ell=-1.
\end{cases}
\end{equation}
Since both sign classes occur and $\eta<1$, each $v_i$ has a positive and a negative coordinate. Thus $v_i$ is nonconstant and $u_i\ne0$.

The midpoint of the interval for coordinate $\ell$ is $(1-\eta/2)\sigma_\ell$. Therefore
\[
\|v_i-(1-\eta/2)\sigma^{\mathsf T}\|\leq\frac{\sqrt n\,\eta}{2}.
\]
As $H$ is an orthogonal projection, all $u_i$ lie in the ball with centre and radius
\[
c=(1-\eta/2)\sigma^{\mathsf T}H,
\qquad R=\frac{\sqrt n\,\eta}{2}.
\]
Moreover,
\[
\|\sigma^{\mathsf T}H\|^2
=n-\frac{(n-2m)^2}{n}
=\frac{4m(n-m)}n=\frac{ne}{2}.
\]
Consequently $c\ne0$ and $\|c\|^2=(ne/2)(1-\eta/2)^2$. Since $u_i=\sigma_i x_i$,
\[
|\Cor(P)_{ij}-\sigma_i\sigma_j|
=1-\frac{\langle u_i,u_j\rangle}{\|u_i\|\,\|u_j\|}.
\]
Lemma~\ref{lem:ball} bounds this expression by $2R^2/\|c\|^2$, which is the right side of~\eqref{eq:new}. Taking the maximum proves the result.
\end{proof}

\begin{remark}[The elementary repair]\label{rem:repair}
The separate distance and norm estimates in Kruskal's argument~\cite[pp.~4--8]{Kruskal}, retaining the factor $(1-\eta)^2$, give the weaker but valid bound
\begin{equation}\label{eq:repair}
\dev(\Cor(P))\leq\frac{\eta^2}{e(1-\eta)^2}.
\end{equation}
Indeed,~\eqref{eq:intervals} gives $\|u_i-u_j\|^2\leq n\eta^2$. The variance identity
\[
\|vH\|^2=\frac1n\sum_{\ell<j}(v_\ell-v_j)^2
\]
and the $m(n-m)$ cross-class coordinate pairs give
\[
\|u_i\|^2\geq\frac{4m(n-m)}n(1-\eta)^2.
\]
Now apply
\[
1-\frac{\langle u_i,u_j\rangle}{d_i d_j}
=\frac{\|u_i-u_j\|^2-(d_i-d_j)^2}{2d_i d_j},
\qquad d_i=\|u_i\|.
\]
The common-ball argument improves~\eqref{eq:repair} strictly when $0<\eta<1$.
\end{remark}

\begin{proof}[Proof of Theorem~\ref{thm:main}]
Set
\[
\kappa(\eta)=\frac{\eta}{e(1-\eta/2)^2},\qquad 0\leq\eta<1.
\]
This function is strictly increasing, since
\[
\kappa'(\eta)=\frac{1+\eta/2}{e(1-\eta/2)^3}>0.
\]
Solving $\kappa(\eta)=1$ gives the smaller root $\gamma(e)$ in~\eqref{eq:gamma}. The larger root exceeds $1$, and $0<\gamma(e)\leq\gamma(2)<1$. Hence $\kappa(\eta)<1$ exactly when $\eta<\gamma(e)$ in the stated interval.

If $\eta_0=0$, the trajectory is constant. Otherwise let $q=\kappa(\eta_0)<1$. Proposition~\ref{prop:bound} and induction show that all iterates are defined, belong to $\E_n$, and satisfy
\[
\eta_{t+1}\leq\kappa(\eta_t)\eta_t\leq q\eta_t,
\qquad \eta_t\leq\eta_0.
\]
In particular $\eta_t\to0$. The same proposition gives $\eta_{t+1}\leq C\eta_t^2$ with $C$ as in~\eqref{eq:Cq}. Setting $b_t=C\eta_t$ yields $b_{t+1}\leq b_t^2$, hence $b_t\leq q^{2^t}$. This proves~\eqref{eq:rate}. If $\eta_t>0$ for all $t$, divide~\eqref{eq:new} by $\eta_t^2$ and use $\eta_t\to0$ to obtain~\eqref{eq:asymptotic}.
\end{proof}

\begin{remark}[General bounded initial matrices]
The proof of Proposition~\ref{prop:bound} uses only the entry bounds $[-1,1]$ and proximity to $\sigma\sigma^{\mathsf T}$, not symmetry, positive semidefiniteness, or unit diagonal. Thus Theorem~\ref{thm:main} also holds for a real initial matrix $P_0\in[-1,1]^{n\times n}$ satisfying the same deviation condition, with the row-correlation convention~\eqref{eq:map}: the first iterate is a correlation matrix and the same estimates apply thereafter. Transposition gives the corresponding statement for column correlations.
\end{remark}

\subsection*{Comparison of the sufficient thresholds}
The elementary repair~\eqref{eq:repair} yields
\[
\alpha(e)=\frac{2e+1-\sqrt{4e+1}}{2e},
\]
whereas the original estimate~\eqref{eq:old} would yield Kruskal's stated threshold~\cite[pp.~1, 3]{Kruskal},
\[
\beta(e)=\frac{\sqrt{1+4e^2}-1}{2e}.
\]
For $0<e\leq2$, one has
\begin{equation}\label{eq:comparison}
\alpha(e)<\gamma(e)<\beta(e).
\end{equation}
The first inequality follows by comparing the increasing contraction factors in~\eqref{eq:repair} and~\eqref{eq:new}. For the second, $\beta(e)\leq\beta(2)<4/5$ and
\[
(1-\eta/2)^2<1-\eta^2\qquad(0<\eta<4/5).
\]
Thus $\kappa(\beta(e))>1$, which gives $\gamma(e)<\beta(e)$. In the balanced case,
\[
\alpha(2)=\frac12,\qquad
\gamma(2)=3-\sqrt5,\qquad
\beta(2)=\frac{\sqrt{17}-1}{4}.
\]
The new estimate enlarges the neighbourhood certified by the elementary repair, not the neighbourhood claimed from the invalid estimate. The counterexamples do not establish failure of convergence between $\gamma(e)$ and $\beta(e)$.

\section{Invariant families and counterexamples in every dimension at least four}\label{sec:families}
The even- and odd-dimensional examples can be treated by one construction. When opposite groups have unequal multiplicities, the row means need not vanish; the centering terms must therefore be retained.

\begin{proposition}[A common scalar recurrence]\label{prop:family}
Let $r\geq1$ and $p,q\geq0$ be integers with $p+q\geq1$, and put $n=2r+p+q$. For unit vectors $u,v$ with $\langle u,v\rangle=a\in(0,1)$, let $Q_{r,p,q}(a)$ be the Gram matrix of
\begin{equation}\label{eq:groups}
\underbrace{u,\ldots,u}_{r},\quad
\underbrace{v,\ldots,v}_{p},\quad
\underbrace{-u,\ldots,-u}_{r},\quad
\underbrace{-v,\ldots,-v}_{q}.
\end{equation}
An empty group is omitted. The target pattern assigns $+1$ to the first two groups and $-1$ to the last two. Define
\begin{equation}\label{eq:AB}
A=2r,\qquad B=p+q-\frac{(p-q)^2}{n},\qquad L=A+B.
\end{equation}
Then $A,B>0$, $\dev(Q_{r,p,q}(a))=1-a$, and
\begin{equation}\label{eq:scalar}
\Cor(Q_{r,p,q}(a))=Q_{r,p,q}(F(a)),
\qquad F(a)=\frac{La}{\sqrt{(A+Ba^2)(B+Aa^2)}}.
\end{equation}
Every trajectory with $0<a_0<1$ converges to its target pattern, and
\begin{equation}\label{eq:family-coefficient}
\lim_{t\to\infty}\frac{\eta_{t+1}}{\eta_t^2}
=\frac{2AB}{L^2}>0,
\qquad \eta_t=1-a_t.
\end{equation}
The class sizes are $r+p$ and $r+q$, and the associated parameter satisfies $e=2L/n$.
\end{proposition}

\begin{proof}
The Gram construction gives positive semidefiniteness and unit diagonal. Since both $u$ and a vector of type $v$ or $-v$ occur, the nonzero deviations from the target equal $1-a$. Put $d=p-q$. A row of type $u$ has entries $1,a,-1,-a$ in the four groups, respectively; a row of type $v$ has entries $a,1,-a,-1$. Their means are $da/n$ and $d/n$, respectively. Subtracting the mean terms from their squared norms and inner product gives
\begin{align*}
\|x_u\|^2&=2r+(p+q)a^2-\frac{d^2a^2}{n}=A+Ba^2,\\
\|x_v\|^2&=2ra^2+p+q-\frac{d^2}{n}=B+Aa^2,\\
\langle x_u,x_v\rangle&=na-\frac{d^2a}{n}=La.
\end{align*}
Rows of negative type are the negatives of their positive counterparts, including after centering. Here $A>0$, and $B>0$ because $d^2\leq(p+q)^2<n(p+q)$. Thus the update is well defined and has exactly the form~\eqref{eq:scalar}, including when one of the $v$-groups is absent.

For $0<a<1$, each of $A+Ba^2$ and $B+Aa^2$ is less than $L$, so $F(a)>a$. Also,
\begin{equation}\label{eq:detidentity}
(A+Ba^2)(B+Aa^2)-L^2a^2=AB(1-a^2)^2>0,
\end{equation}
which gives $F(a)<1$. The parameters therefore increase to a limit in $(0,1]$. Continuity and~\eqref{eq:detidentity} show that the only possible limit is $1$.

Rationalizing $1-F(a)$ using~\eqref{eq:detidentity} gives
\begin{equation}\label{eq:ratio}
\frac{1-F(a)}{(1-a)^2}
=\frac{AB(1+a)^2}
{(A+Ba^2)(B+Aa^2)(1+F(a))}.
\end{equation}
Letting $a\to1$ proves~\eqref{eq:family-coefficient}. Finally, the difference of class sizes is $d$, so $8(r+p)(r+q)/n^2=2(1-d^2/n^2)=2L/n$.
\end{proof}

\begin{corollary}[All dimensions at least three]\label{cor:dimensions}
The family in Proposition~\ref{prop:family} exists in every dimension $n\geq3$. In even dimension $n=2k\geq4$, choose integers $r,s\geq1$ with $r+s=k$ and set $p=q=s$; then $A=2r$, $B=2s$, and $e=2$. In odd dimension $n=2k+1\geq3$, choose integers $r\geq1$, $s\geq0$ with $r+s=k$ and set $p=s+1$, $q=s$; then
\[
A=2r,\qquad B=2s+1-\frac1n,\qquad
e=2\left(1-\frac1{n^2}\right).
\]
All these trajectories converge, and their quadratic asymptotic coefficients are positive.
\end{corollary}

\begin{proof}[Proof of the counterexample assertion in Theorem~\ref{thm:examples}]
Set $a=2/3$ in Proposition~\ref{prop:family} and write
\[
\tau=\frac{AB}{L^2}.
\]
Equation~\eqref{eq:detidentity} gives
\begin{equation}\label{eq:actual-error}
\eta=\frac13,\qquad
\eta^+=1-F(2/3)
=1-\left(1+\frac{25\tau}{36}\right)^{-1/2}.
\end{equation}
The bound proposed in~\eqref{eq:old} is $1/(8e)$.

If $n=2k\geq4$, take $r=\lfloor k/2\rfloor$, $s=\lceil k/2\rceil$, and $p=q=s$. Then
\[
e=2,\qquad \tau=\frac{rs}{(r+s)^2}\geq\frac29.
\]
For even $k$ the ratio is $1/4$; for odd $k\geq3$ it is $(k^2-1)/(4k^2)\geq2/9$. Hence
\begin{equation}\label{eq:even-counter}
\eta^+\geq1-\frac{18}{\sqrt{374}}>\frac1{16}=\frac1{8e}.
\end{equation}
The strict inequality follows by squaring $288<15\sqrt{374}$.

If $n\geq5$ is odd, write $n=4\ell+1$ or $n=4\ell+3$, where $\ell\geq1$. Choose $(r,s)=(\ell,\ell)$ in the first case and $(r,s)=(\ell+1,\ell)$ in the second, and put $p=s+1$, $q=s$. Direct substitution into~\eqref{eq:AB} gives, in both cases,
\[
\tau=\frac14-\frac{1}{16(2\ell+1)^2}\geq\frac{35}{144}.
\]
Also $e=2(1-n^{-2})\geq48/25$. Therefore
\begin{equation}\label{eq:odd-counter}
\eta^+\geq1-\frac{72}{\sqrt{6059}}
>\frac{25}{384}\geq\frac1{8e}.
\end{equation}
For the strict inequality, $359^2\cdot6059-27648^2=16478075>0$. Thus~\eqref{eq:counterclaim} holds in every dimension $n\geq4$. Convergence of all the constructed trajectories follows from Proposition~\ref{prop:family}.
\end{proof}

\begin{remark}
The construction also gives a convergent family in dimension three, with $r=1$, $p=1$, $q=0$, $A=2$, and $B=2/3$. Its asymptotic coefficient is $3/8$. This establishes attainment of quadratic order in dimension three; it does not establish failure of~\eqref{eq:old} there.
\end{remark}

\begin{corollary}[Positive-definite counterexamples]\label{cor:pd}
For every $n\geq4$, there are a fixed sign pattern and a nonempty relatively open subset of the positive-definite correlation matrices on which the original estimate~\eqref{eq:old} fails. Every trajectory starting in this subset converges to that pattern.
\end{corollary}

\begin{proof}
Fix a matrix $P$ constructed in the proof of~\eqref{eq:counterclaim}, together with its pattern and parameter $e$. Then $\dev(P)=1/3$ and $e\geq48/25$. In particular,
\[
\kappa(1/3)=\frac{12}{25e}\leq\frac14<1,
\]
so $\dev(P)<\gamma(e)$. Define
\[
P_\varepsilon=(1-\varepsilon)P+\varepsilon I_n,
\qquad 0<\varepsilon<1.
\]
This matrix has unit diagonal and is positive definite, since
\[
z^{\mathsf T}P_\varepsilon z
\geq\varepsilon\|z\|^2>0\qquad(z\ne0).
\]
The centered row norms of $P$ are nonzero. Hence the correlation map, and therefore the function
\[
G(Q)=\dev(\Cor(Q))-
\frac{\dev(Q)^2}{e(1-\dev(Q)^2)},
\]
are continuous in a neighbourhood of $P$ where $\dev(Q)<1$. The constructed counterexample satisfies $G(P)>0$. As $P_\varepsilon\to P$, both $G(P_\varepsilon)>0$ and $\dev(P_\varepsilon)<\gamma(e)$ hold for all sufficiently small positive $\varepsilon$. Positive definiteness and these strict inequalities persist on a neighbourhood of $P_\varepsilon$ relative to the affine space of real symmetric matrices with unit diagonal. This gives the asserted relatively open set. Theorem~\ref{thm:main} proves convergence for every matrix in it.
\end{proof}

For an explicit positive-definite example, let $P_0$ be~\eqref{eq:four} and set
\[
\widetilde P=\frac{99}{100}P_0+\frac1{100}I_4.
\]
All its eigenvalues are at least $1/100$. Its centered rows have the common squared norm $114051/40000$, and direct calculation gives
\[
\dev(\widetilde P)=\frac{17}{50},\qquad
\dev(\Cor(\widetilde P))=\frac{8980}{114051}
>\frac{289}{4422}
=\frac{(17/50)^2}{2(1-(17/50)^2)}.
\]
Thus the failure of the original estimate is not confined to singular matrices or to matrices with repeated rows.

\section{Exact trajectories and sharpness}\label{sec:sharp}
When all four groups have the same size, the scalar recurrence can be solved explicitly. Let
\begin{equation}\label{eq:Pa}
P(a)=\begin{pmatrix}
1&a&-1&-a\\
a&1&-a&-1\\
-1&-a&1&a\\
-a&-1&a&1
\end{pmatrix},\qquad 0<a<1,
\end{equation}
and let $J_h$ be the $h\times h$ all-ones matrix. Taking $h$ copies of each of $u$, $v$, $-u$, and $-v$ gives four groups of equal size and hence dimension $n=4h$. This requirement is stronger than a balanced sign partition, which requires only that the two sign classes have equal size. The equal-group construction is
\[
Q_h(a)=P(a)\otimes J_h\in\E_{4h}.
\]

For this family, the rows have zero mean and a common norm, so Pearson iteration reduces to matrix squaring followed by a scalar normalization. The spectral squaring mechanism is classical; see Schwartz~\cite[pp.~260--266]{Schwartz}. The two nonzero eigenvalues of $Q_h(a)$ are $2h(1+a)$ and $2h(1-a)$, and their ratio $(1-a)/(1+a)$ is squared at each step. We use this reduction to obtain exact deviation formulas, violations of Kruskal's estimate, and attainment of the asymptotic coefficient.

\begin{proposition}[An explicitly solvable subfamily]\label{prop:exact}
For $h\geq1$,
\begin{equation}\label{eq:equal-map}
\Cor(Q_h(a))=Q_h\!\left(\frac{2a}{1+a^2}\right).
\end{equation}
If $z_0=(1-a_0)/(1+a_0)$, then its trajectory is given by
\begin{equation}\label{eq:exact}
a_t=\frac{1-z_0^{\,2^t}}{1+z_0^{\,2^t}},
\qquad \eta_t=\frac{2z_0^{\,2^t}}{1+z_0^{\,2^t}}.
\end{equation}
Its deviation satisfies
\begin{equation}\label{eq:exact-error}
\eta_{t+1}=\frac{\eta_t^2}{2-2\eta_t+\eta_t^2},
\qquad
\lim_{t\to\infty}\frac{\eta_{t+1}}{\eta_t^2}=\frac12.
\end{equation}
The original estimate~\eqref{eq:old} fails for every $1/3<a<1$ in this family.
\end{proposition}

\begin{proof}
Take $r=p=q=h$ in Proposition~\ref{prop:family}; then $A=B=2h$, giving~\eqref{eq:equal-map}. The transformation $z=(1-a)/(1+a)$ satisfies $z_{t+1}=z_t^2$. Solving this recurrence and using $\eta_t=1-a_t$ proves~\eqref{eq:exact}. Substituting $a_t=1-\eta_t$ into~\eqref{eq:equal-map} gives~\eqref{eq:exact-error}.

Here $e=2$. For $0<\eta<2/3$,
\[
2(1-\eta^2)-(2-2\eta+\eta^2)=\eta(2-3\eta)>0.
\]
Both denominators are positive. Thus the actual deviation in~\eqref{eq:exact-error} exceeds $\eta^2/[2(1-\eta^2)]$. Since $\eta=1-a$, this is exactly the asserted interval.
\end{proof}

\begin{proof}[Completion of the proof of Theorem~\ref{thm:examples}]
Let $a_0\to1$ in Proposition~\ref{prop:exact}. This gives violations of~\eqref{eq:old} arbitrarily close to the balanced fixed point in each dimension $4h$. Taking $1-a_0<\gamma(2)$ places these examples within Theorem~\ref{thm:main}, and~\eqref{eq:exact-error} attains its coefficient $1/e=1/2$.

In every dimension $n\geq3$, choose one family from Corollary~\ref{cor:dimensions}. Its asymptotic coefficient $K=2AB/L^2$ is strictly positive. For any real $\nu>2$,
\[
\frac{\eta_{t+1}}{\eta_t^\nu}
=\frac{\eta_{t+1}}{\eta_t^2}\,\eta_t^{2-\nu}
\longrightarrow+\infty.
\]
Hence no fixed finite constant can bound $\eta_{t+1}$ by that constant times $\eta_t^\nu$ for all sufficiently large $t$ along these trajectories. The initial parameter can be chosen arbitrarily close to $1$, so the same conclusion holds for examples initially inside the local neighbourhood of Theorem~\ref{thm:main}. This proves attainment of quadratic order.
\end{proof}

\subsection*{Sharpness and dimensional scope}
The invariant families establish that quadratic order is attained in every dimension $n\geq3$. In dimensions divisible by four, the equal-group family attains the coefficient $1/e=1/2$ for balanced partitions. Thus any asymptotic quadratic upper bound valid for all trajectories near such a fixed point must have coefficient at least $1/2$. These are sharpness statements for the convergence order and asymptotic coefficient; optimality of the full finite-error bound~\eqref{eq:new} and of the sufficient threshold~\eqref{eq:gamma} is a separate question.

The construction proves failure of Kruskal's estimate in every dimension $n\geq4$. The stronger conclusion of failure arbitrarily close to the fixed point is established here in dimensions divisible by four. In dimension three, the family proves attainment of quadratic order without furnishing a counterexample to the original estimate. Dimension two is exceptional: the two centered rows of a nondegenerate $2\times2$ correlation matrix are nonzero negatives of one another, so one iteration gives the sole nondegenerate sign-pattern fixed point.

Global convergence of every admissible trajectory in dimension three is established in~\cite{AH}. The present results complement that theorem with local quantitative bounds in arbitrary dimension, counterexamples to Kruskal's estimate for every $n\geq4$, and invariant families attaining the stated convergence rates. Global convergence alone does not decide the validity of a particular one-step estimate.

\end{document}